\documentclass{article}
\usepackage[latin1]{inputenc}
\usepackage[T1]{fontenc}
\usepackage{amsmath,amssymb}
\usepackage{amsfonts,amsthm}

\usepackage{geometry}
\usepackage{color}
   \definecolor{cites}{rgb}{0.75 , 0.00 , 0.00}  
   \definecolor{urls} {rgb}{0.00 , 0.00 , 1.00}  
   \definecolor{links}{rgb}{0.00 , 0.00 , 0.5}   
  \definecolor{gray}{rgb}{0.5,0.5,.5}
\usepackage{graphicx}
\usepackage{cancel}
\usepackage[plainpages=false,colorlinks=true, citecolor=cites, urlcolor=urls, linkcolor=links, breaklinks=true]{hyperref}

\allowdisplaybreaks

\newcommand{\C}{\mathbb{C}}

\newcommand{\N}{\mathbb{N}}

\newcommand{\eps}{\epsilon}
\renewcommand{\epsilon}{\varepsilon}

\newcommand{\vertiii}[1]{{\left\vert\kern-0.25ex\left\vert\kern-0.25ex\left\vert #1
    \right\vert\kern-0.25ex\right\vert\kern-0.25ex\right\vert}}
\newcommand{\vertiiis}[1]{{\vert\kern-0.25ex\vert\kern-0.25ex\vert #1
    \vert\kern-0.25ex\vert\kern-0.25ex\vert}}

\DeclareMathOperator{\VMO}{VMO}

\DeclareMathOperator{\Real}{Re}
\renewcommand{\Re}{\Real}
\DeclareMathOperator{\Imag}{Im}
\renewcommand{\Im}{\Imag}

\newcommand{\from}{\colon}

\providecommand{\scpr}[2]{\left\langle #1, #2 \right\rangle}
\renewcommand{\sp}{\scpr}
\providecommand{\abs}[1]{\left\lvert#1\right\rvert}
\providecommand{\norm}[1]{\left\lVert#1\right\rVert}
\providecommand{\set}[1]{\left\{ #1\right\}}

\newtheorem{thm}{Theorem}
\newtheorem{mthm}[thm]{Theorem}
\newtheorem{lem}[thm]{Lemma}
\newtheorem{prop}[thm]{Proposition}
\newtheorem{cor}[thm]{Corollary}
\newtheorem*{cor*}{Corollary}

\theoremstyle{definition}
\newtheorem{defn}[thm]{Definition}

\theoremstyle{remark}
\newtheorem{rem}[thm]{Remark}

\numberwithin{equation}{section}

\allowdisplaybreaks
\begin{document}
\title{\bf A Product Expansion for Toeplitz Operators on the Fock Space}
\author{Raffael Hagger\footnote{Institut f\"ur Analysis, Leibniz Universit\"at, 30167 Hannover, Germany, raffael.hagger@math.uni-hannover.de}}
\maketitle
\vspace{-0.4cm}
\begin{abstract}
We study the asymptotic expansion of the product of two Toeplitz operators on the Fock space. In comparison to earlier results we require significantly less derivatives and get the expansion to arbitrary order. This, in particular, improves a result of Borthwick related to Toeplitz quantization. In addition, we derive an intertwining identity between the Berezin star product and the sharp product.


\end{abstract}

\section{Introduction} \label{introduction}

It is a long standing open question in the theory of Toeplitz operators to determine when the product of two Toeplitz operators is again a Toeplitz operator. For the Hardy space this was famously solved by Brown and Halmos \cite{BroHa}, basically saying that this only happens in trivial cases. On Bergman and Fock spaces the situation is different. For the Bergman space over the unit disc some examples are given in \cite{AhCu} for instance. For the Fock space there are plenty of examples. For polynomials $p$ and $q$ in $z$ and $\bar{z}$ on the $n$-particle Fock space it is well-known and easy to check that the following expansion holds:
\begin{equation} \label{product_expansion_poly}
T_pT_q = \sum\limits_{\alpha \in \N_0^n} \frac{(-1)^{\abs{\alpha}}}{\alpha!} T_{(\partial^\alpha p)(\bar{\partial}^\alpha q)},
\end{equation}
where we used the standard multi-index notations $\abs{\alpha} = \alpha_1 + \ldots + \alpha_n$ and $\partial^\alpha = \frac{\partial^{\abs{\alpha}}}{\partial z_1^{\alpha_1} \cdots \partial z_n^{\alpha_n}}$ (likewise for $\bar{\partial}^\alpha$). Note that for polynomials the above sum is finite and the right-hand side is again a Toeplitz operator. In fact, this expansion plays an important role in Rieffel's deformation quantization \cite{Rieffel}. Indeed, if we set $q_j = \Re(z_j) = \frac{z_j+\bar{z}_j}{2}$ and $p_j = \Im(z_j) = \frac{z_j-\bar{z}_j}{2i}$, Equation \eqref{product_expansion_poly} yields $[T_{p_j},T_{p_k}] = [T_{q_j},T_{q_k}] = 0$ and $[T_{p_j},T_{q_k}] = \frac{i}{2}\delta_{jk}$, where $[\cdot,\cdot]$ denotes the commutator of two operators and $\delta_{j,k}$ is the Kronecker delta. That is, the operators $T_{p_1}, \ldots, T_{p_n}, T_{q_1}, \ldots, T_{q_n}$ satisfy the canonical commutation relations (up to an irrelevant constant). Moreover, to first order in the expansion, we have
\[[T_f,T_g] = T_{i\{f,g\}} + \ldots,\]
where $\{f,g\}$ denotes the Poisson bracket of $f$ and $g$. This can be seen as an instance of the general principle ``replace commutators by Poisson brackets'' in quantization. In this sense Toeplitz operators serve as a realization of Rieffel's deformation quantization. It is therefore an interesting question for which symbols (other than polynomials) the expansion \eqref{product_expansion_poly} holds. To make this a little bit more precise, let us give some definitions first before we proceed with known results on the matter.

Let $\set{\mu_t}_{t > 0}$ denote the family of Gaussian probability measures given by
\[\mathrm{d}\mu_t(z) = \frac{1}{(\pi t)^ n}e^{-\frac{\abs{z}^2}{t}} \, \mathrm{d}v(z),\]
where $\mathrm{d}v$ denotes the standard Lebesgue volume form on $\C^n$. The Fock space $H^2_t$ is then defined as
\[H^2_t := \set{f \from \C^n \to \C : f \text{ is entire and } \int_{\C^n} \abs{f}^2 \, \mathrm{d}\mu_t < \infty}.\]
An orthonormal basis of $H^2_t$ is given by the monomials $e_{\alpha}^{(t)}(z) := \frac{z^{\alpha}}{\sqrt{\alpha !t^{\abs{\alpha}}}}$ ($\alpha \in \N_0^n$). As $H^2_t$ is a closed subspace of $L_t^2 := \set{f \from \C^n \to \C : \int_{\C^n} \abs{f}^2 \, \mathrm{d}\mu_t < \infty}$, there exists an orthogonal projection $P^{(t)} \from L^2_t \to H^2_t$. More explicitly, we have
\[[P^{(t)}f](z) = \int_{\C^n} f(w)e^{\frac{\sp{z}{w}}{t}} \, \mathrm{d}\mu_t.\]
For bounded functions $f$, called symbols, we consider the corresponding Toeplitz operator $T_f^{(t)} := P^{(t)}M_f \from H^2_t \to H^2_t$, where $M_f$ denotes the operator of multiplication by $f$. If we factor in the deformation (or weight) parameter $t$, the expansion \eqref{product_expansion_poly} reads as
\begin{equation} \label{product_expansion}
T_f^{(t)}T_g^{(t)} = \sum\limits_{\alpha \in \N_0^n} \frac{(-t)^{\abs{\alpha}}}{\alpha!} T_{(\partial^\alpha f)(\bar{\partial}^\alpha g)}^{(t)}.
\end{equation}
Comparing with the canonical commutation relations, we see that $t$ can be interpreted as the reduced Planck constant $\hbar$. For general symbols $f$ and $g$ this sum may not converge in any meaningful way. To make sense of it, we will look at it as an asymptotic expansion, that is, for fixed $k \in \N$ we ask whether
\begin{equation} \label{asymptotic_expansion}
\lim\limits_{t \to 0} \frac{1}{t^k}\norm{T_f^{(t)}T_g^{(t)} - \sum\limits_{\abs{\alpha} \leq k} \frac{(-t)^{\abs{\alpha}}}{\alpha!} T_{(\partial^\alpha f)(\bar{\partial}^\alpha g)}^{(t)}} = 0.
\end{equation}
For $k = 0$ and $k = 1$ there are several results: Coburn \cite{Coburn} showed that \eqref{asymptotic_expansion} holds if $f$ and $g$ are the sum of trigonometric polynomials and $(2n+6)$-times continuously differentiable functions with compact support. This was extented by Borthwick \cite{Borthwick} to symbols that have $4n+6$ bounded derivatives but possibly unbounded support. For $k = 0$ it was shown in \cite{BaCo} that \eqref{asymptotic_expansion} even holds for bounded uniformly continuous functions. Moreover, a counterexample was given to show that
\begin{equation} \label{asymptotic_expansion_zero}
\lim\limits_{t \to 0} \norm{T_f^{(t)}T_g^{(t)} - T_{fg}^{(t)}} = 0
\end{equation}
(i.e.~\eqref{asymptotic_expansion} for $k = 0$) does not hold for general $f$ and $g$. Recently these results were extended further in \cite{BaCoHa} as follows. \eqref{asymptotic_expansion_zero} holds for arbitrary uniformly continuous functions $f$ and $g$ even if the corresponding Toeplitz operators are unbounded (the semi-commutator in \eqref{asymptotic_expansion_zero} is automatically bounded). Moreover, one may choose one of the two functions $f$ or $g$ to be an arbitrary, possibly discontinuous, bounded function. The uniform continuity may also be replaced by bounded $\VMO$ and this algebra is the largest $C^*$-algebra with this property (see \cite{BaCoHa} for definitions and the precise statement). To round things up, an example of two bounded continuous functions with high oscillation was given, where \eqref{asymptotic_expansion_zero} is violated. This suggests that high oscillation somewhat prohibits \eqref{asymptotic_expansion}. However, this would need some more elaboration. It is worth mentioning that Toeplitz quantization has also been studied on different domains like the unit ball, pseudoconvex domains or compact symplectic manifolds and similar results have been obtained under varying differentiability assumptions \cite{BaMaMaPi,BaHaVa,Berezin2,Berezin3,BoLeUp,Englis,EnUp,EnUp2,KliLe,MaMa}.

Bauer \cite[Theorem 16]{Bauer} showed that the series \eqref{product_expansion} converges to a bounded Toeplitz operator in case $f$ and $g$ are in the range of the heat transform. The heat transform (often also called Berezin transform \cite{Berezin2}) $f^{(t)}$ of a function $f \in L^{\infty}(\C^n)$ is defined as
\[f^{(t)} \from \C^n \to \C, \quad f^{(t)}(z) := \frac{1}{(t\pi)^n} \int_{\C^n} f(w)e^{-\frac{1}{t}\abs{w-z}^2} \, \mathrm{d}w = \frac{1}{\pi^n} \int_{\C^n} f(\sqrt{t}w + z)e^{-\abs{w}^2} \, \mathrm{d}w.\]
Even though this formula of Bauer is an impressive result, all the above results for $k \geq 1$ are a little bit unsatisfactory as there is always the need of a ridiculous amount of derivatives, where the expression \eqref{asymptotic_expansion} only needs $k$ derivatives to make sense syntactically. For instance, in case $n = 1$ Borthwick \cite{Borthwick} needs 10 bounded derivatives, where the formula itself only involves one single derivative. This is exactly the motivation for this paper. We give a quick and elementary proof that at most $2k$ bounded uniformly continuous derivatives are needed for \eqref{asymptotic_expansion}. The set of functions which are $2k$-times continuously differentiable and for which all partial derivatives up to order $2k$ are bounded and uniformly continuous will be denoted by $C^{2k}_{buc}(\C^n)$. For $k = 0$, these are just the bounded uniformly continuous functions.

\begin{mthm} \label{mthm}
Let $k \geq 0$ and $f,g \in C^{2k}_{buc}(\C^n)$. Then the following product expansion holds:
\[T_f^{(t)}T_g^{(t)} = \sum\limits_{\substack{\alpha \in \N_0^n \\ \abs{\alpha} \leq k}} \frac{(-t)^{\abs{\alpha}}}{\alpha !} T_{(\partial^{\alpha}f)(\bar{\partial}^{\alpha}g)}^{(t)} + o(t^k)\]
in the operator norm sense as $t \to 0$.
\end{mthm}

As the Poisson bracket in complex coordinates is given by $\{f,g\} = i\sum\limits_{j = 1}^n \partial_j f \bar{\partial}_j g - \bar{\partial}_j f \partial_j g$, this implies the following corollary, which directly improves the result of Borthwick \cite[Theorem 4.5]{Borthwick}:

\begin{cor} \label{cor}
Let $f,g \in C^2_{buc}(\C^n)$. Then
\[\norm{[T_f^{(t)},T_g^{(t)}] - itT_{\{f,g\}}^{(t)}} = o(t)\]
as $t \to 0$.
\end{cor}

An example in \cite{BaCoHa} shows that at least in the case $k = 0$ some control of the oscillation is needed. However, it is not clear whether uniform continuity or even more than $k$ derivatives are necessary for Theorem \ref{mthm} if $k \geq 1$.

We also derive the following interesting identity. It can be understood as an intertwining between the sharp product and the Berezin star product $f \ast_B g := \sum\limits_{\alpha \in \N_0^n} \frac{t^{\abs{\alpha}}}{\alpha !} (\bar{\partial}^{\alpha}f)(\partial^{\alpha}g)$ by the heat transform (see \cite[Section 4]{EnglisBook}).

\begin{mthm} \label{mthm2}
Let $k \geq 0$ and $f,g \in C^{2k}_{buc}(\C^n)$. Then the following identity holds:
\[\sum\limits_{\substack{\alpha \in \N_0^n \\ \abs{\alpha} \leq k}} \frac{(-t)^{\abs{\alpha}}}{\alpha !} \left((\partial^{\alpha}f)(\bar{\partial}^{\alpha}g)\right)^{(t)} = \sum\limits_{\substack{\alpha \in \N_0^n \\ \abs{\alpha} \leq k}} \frac{t^{\abs{\alpha}}}{\alpha !} (\bar{\partial}^{\alpha}f^{(t)})(\partial^{\alpha}g^{(t)}) + o(t^k)\]
uniformly on $\C^n$ as $t \to 0$.
\end{mthm}

The outline of this paper is as follows. In Section 2 and 3 we give a proof of Theorem \ref{mthm} and Theorem \ref{mthm2}, respectively. In Section 4 we list some open problems which are related to this paper.

\section{Proof of Theorem 1}

The proof of Theorem \ref{mthm} only relies on Bauer's result and an ordinary Taylor expansion in complex coordinates. Let us define the following family of sharp products of $C^{\infty}$-functions.

\begin{defn} \label{definition}
Let $t > 0$. For $f,g \in C^{\infty}(\C^n)$ we define the formal power series
\begin{equation} \label{sharp_product}
f \sharp_t g := \sum\limits_{\alpha \in \N_0^n} \frac{(-t)^{\abs{\alpha}}}{\alpha !} (\partial^{\alpha}f)(\bar{\partial}^{\alpha}g).
\end{equation}
\end{defn}

Note that this is only a formal series and no convergence is guaranteed. However, by the mentioned result of Bauer \cite[Theorem 16]{Bauer}, for functions in the range of the heat transform the series converges and we have:

\begin{prop} \label{thm_BaCo}
(\cite[Theorem 1]{Bauer,BaCo})\\
Let $f,g \in L^{\infty}(\C^n)$. Then the power series $[f^{(t)} \sharp_t g^{(t)}](z)$ converges for all $z \in \C^n$ to
\[[f^{(t)} \sharp_t g^{(t)}](z) = \frac{1}{\pi^{2n}} \int_{\C^n} \int_{\C^n} f(\sqrt{t}v+z)g(\sqrt{t}w+z)e^{-\bar{v}w-\abs{v}^2-\abs{w}^2} \, \mathrm{d}v \, \mathrm{d}w.\]
Moreover, $f^{(t)} \sharp_t g^{(t)}$ is bounded and $T_{f^{(t)}}^{(t)}T_{g^{(t)}}^{(t)} = T_{f^{(t)} \sharp_t g^{(t)}}^{(t)}$.
\end{prop}

For the proof of Theorem \ref{mthm} we need two lemmas to expand $f^{(t)}$, $g^{(t)}$ and $f^{(t)} \sharp_t g^{(t)}$ in terms of $f$, $g$ and their derivatives.

\begin{lem} \label{lem1}
Let $k \geq 1$ and $f \in C^{2k}_{buc}(\C^n)$. Then we have the following expansion:
\begin{equation} \label{heat_expansion}
f^{(t)}(z) = \sum\limits_{\substack{\alpha \in \N_0^n \\ \abs{\alpha} \leq k}} \frac{t^{\abs{\alpha}}}{\alpha !}\partial^{\alpha}\bar{\partial}^{\alpha}f(z) + o(t^k)
\end{equation}
uniformly for all $z \in \C^n$ as $t \to 0$.
\end{lem}

This expansion is sort of ``obvious'' if we think of $f^{(\cdot)}$ as the solution of the heat equation
\[\partial_tf^{(\cdot)} = \Delta f^{(\cdot)} := \sum\limits_{k = 1}^n \partial_k\bar{\partial}_kf^{(\cdot)}\]
with initial condition $f^{(0)} = f$, i.e.~$f^{(t)} = e^{t\Delta}f$. Nevertheless, we provide a proof for completeness:

\begin{proof}
Taylor expanding $f(\sqrt{t}w+z)$ around $z$ yields
\begin{equation} \label{Taylor_expansion}
f(\sqrt{t}w+z) = \sum\limits_{\abs{\alpha}+\abs{\beta} \leq 2k} \frac{1}{\alpha !\beta !} \partial^{\alpha}\bar{\partial}^{\beta}f(z)w^{\alpha}\bar{w}^{\beta}t^{\frac{\abs{\alpha}+\abs{\beta}}{2}} + \sum\limits_{\abs{\alpha}+\abs{\beta} = 2k} h_{\alpha,\beta}(\sqrt{t}w+z)w^{\alpha}\bar{w}^{\beta}t^k
\end{equation}
with $h_{\alpha,\beta}(\sqrt{t}w+z) \to 0$ as $t \to 0$. Here, the remainders $h_{\alpha,\beta}(\sqrt{t}w+z)$ in Lagrange form are
\[h_{\alpha,\beta}(\sqrt{t}w+z) = \frac{1}{\alpha !\beta !} \partial^{\alpha}\bar{\partial}^{\beta}\left(f(z + \xi\sqrt{t}w) - f(z)\right)\]
for some $\xi \in (0,1)$. As all derivatives of $f$ are assumed to be bounded and uniformly continuous, the remainders are bounded and the convergence $h_{\alpha,\beta}(\sqrt{t}w+z) \to 0$ is uniform in $z$. The expansion for $f^{(t)}$ now follows by a direct calculation:
\begin{align*}
f^{(t)}(z) &= \frac{1}{\pi^n} \int_{\C^n} f(\sqrt{t}w + z)e^{-\abs{w}^2} \, \mathrm{d}w\\
&= \sum\limits_{\abs{\alpha}+\abs{\beta} \leq 2k} \frac{1}{\alpha !\beta !} \partial^{\alpha}\bar{\partial}^{\beta}f(z) t^{\frac{\abs{\alpha}+\abs{\beta}}{2}} \frac{1}{\pi^n} \int_{\C^n} w^{\alpha}\bar{w}^{\beta} e^{-\abs{w}^2} \, \mathrm{d}w + o(t^k)\\
&= \sum\limits_{\abs{\alpha}+\abs{\beta} \leq 2k} \frac{1}{\alpha !\beta !} \partial^{\alpha}\bar{\partial}^{\beta}f(z) t^{\frac{\abs{\alpha}+\abs{\beta}}{2}} \beta !\delta_{\alpha,\beta} + o(t^k)\\
&= \sum\limits_{\abs{\alpha} \leq k} \frac{t^{\abs{\alpha}}}{\alpha !}\partial^{\alpha}\bar{\partial}^{\alpha}f(z) + o(t^k),
\end{align*}
where we used dominated convergence for the remainder and the fact that $\set{e_{\alpha}(w) = \frac{w^{\alpha}}{\alpha !} : \alpha \in \N_0^n}$ forms an orthonormal basis of $H^2_1$.
\end{proof}

\begin{lem} \label{lem2}
Let $k \geq 0$ and $f,g \in C^{2k}_{buc}(\C^n)$. Then we have the following expansion:
\[[f^{(t)} \sharp_t g^{(t)}](z) = \sum\limits_{\abs{\mu}+\abs{\nu}+\abs{\lambda} \leq k} (-1)^{\abs{\lambda}}\frac{t^{\abs{\mu} + \abs{\nu} + \abs{\lambda}}}{\mu ! \nu ! \lambda !}\left(\partial^{\mu+\lambda}\bar{\partial}^{\mu}f(z)\right)\left(\partial^{\nu}\bar{\partial}^{\nu+\lambda}g(z)\right) + o(t^k)\]
uniformly for all $z \in \C^n$ as $t \to 0$.
\end{lem}

Note that even though plugging \eqref{heat_expansion} directly into \eqref{sharp_product} formally yields the right result, it does not sufficiently take the error term into account. We therefore Taylor expand the integral form of the expression (Proposition \ref{thm_BaCo}) again.

\begin{proof}
Using the Taylor expansion \eqref{Taylor_expansion} for $f$ and $g$ and the same arguments as in Lemma \ref{lem1}, we get
\begin{align*}
[f^{(t)} \sharp_t g^{(t)}](z) &= \frac{1}{\pi^{2n}} \int_{\C^n} \int_{\C^n} f(\sqrt{t}v+z)g(\sqrt{t}w+z)e^{-\bar{v}w-\abs{v}^2-\abs{w}^2} \, \mathrm{d}v \, \mathrm{d}w\\
&= \frac{1}{\pi^{2n}} \sum\limits_{\abs{\alpha}+\abs{\beta}+\abs{\gamma}+\abs{\eps} \leq 2k} \frac{1}{\alpha !\beta !\gamma !\eps !} \left(\partial^{\alpha}\bar{\partial}^{\beta}f(z)\right)\left(\partial^{\gamma}\bar{\partial}^{\eps}g(z)\right) t^{\frac{\abs{\alpha}+\abs{\beta}+\abs{\gamma}+\abs{\eps}}{2}}\\
&\qquad \cdot \int_{\C^n}\int_{\C^n} v^{\alpha}\bar{v}^{\beta}w^{\gamma}\bar{w}^{\eps}e^{-\bar{v}w-\abs{v}^2-\abs{w}^2} \, \mathrm{d}v \, \mathrm{d}w + o(t^k).
\end{align*}
Using polar coordinates and contour integration, the double integral is equal to
\[\prod\limits_{j = 1}^n (-1)\int_0^{\infty} \int_0^{\infty} \int_{\abs{w_j} = 1} \int_{\abs{v_j} = 1} r_j^{\alpha_j+\beta_j+1}s_j^{\gamma_j+\eps_j+1}v_j^{\alpha_j-\beta_j-1}w_j^{\gamma_j-\eps_j-1}e^{-r_js_j\frac{w_j}{v_j} - r_j^2 - s_j^2} \, \mathrm{d}v_j \, \mathrm{d}w_j \, \mathrm{d}r_j \, \mathrm{d}s_j.\]
By Cauchy's integral formula, this is $0$ if $\alpha_j < \beta_j$ or $\gamma_j > \eps_j$ for any $j \in \set{1, \ldots, n}$ and equal to
\begin{align*}
&\prod\limits_{j = 1}^n \frac{2\pi i(-1)^{\alpha_j-\beta_j+1}}{(\alpha_j-\beta_j)!}\int_0^{\infty} \int_0^{\infty} \int_{\abs{w_j} = 1} r_j^{2\alpha_j+1}s_j^{\alpha_j-\beta_j+\gamma_j+\eps_j+1}w_j^{\alpha_j-\beta_j+\gamma_j-\eps_j-1} e^{- r_j^2 - s_j^2} \, \mathrm{d}w_j \, \mathrm{d}r_j \, \mathrm{d}s_j\\
&= \prod\limits_{j = 1}^n \frac{4\pi^2(-1)^{\alpha_j - \beta_j}}{(\alpha_j-\beta_j)!} \delta_{\alpha_j - \beta_j,\eps_j - \gamma_j} \int_0^{\infty} \int_0^{\infty} r_j^{2\alpha_j+1}s_j^{\alpha_j-\beta_j+\gamma_j+\eps_j+1}e^{- r_j^2 - s_j^2} \, \mathrm{d}r_j \, \mathrm{d}s_j\\
&= \prod\limits_{j = 1}^n \frac{4\pi^2(-1)^{\alpha_j - \beta_j}}{(\alpha_j-\beta_j)!} \delta_{\alpha_j - \beta_j,\eps_j - \gamma_j} \int_0^{\infty} \int_0^{\infty} r_j^{2\alpha_j+1}s_j^{2\eps_j+1}e^{- r_j^2 - s_j^2} \, \mathrm{d}r_j \, \mathrm{d}s_j\\
&= \prod\limits_{j = 1}^n \frac{\pi^2(-1)^{\alpha_j - \beta_j}\alpha_j !\eps_j !}{(\alpha_j-\beta_j)!} \delta_{\alpha_j - \beta_j,\eps_j - \gamma_j}\\
&= \pi^{2n}(-1)^{\abs{\alpha}-\abs{\beta}}\frac{\alpha !\eps !}{(\alpha-\beta)!}\delta_{\alpha - \beta,\eps - \gamma}
\end{align*}
otherwise. Choosing $\mu = \beta$, $\nu = \gamma$ and $\lambda = \alpha-\beta = \eps - \gamma$, we get
\[[f^{(t)} \sharp_t g^{(t)}](z) = \sum\limits_{\abs{\mu}+\abs{\nu}+\abs{\lambda} \leq k} (-1)^{\abs{\lambda}}\frac{t^{\abs{\mu} + \abs{\nu} + \abs{\lambda}}}{\mu ! \nu ! \lambda !}\left(\partial^{\mu+\lambda}\bar{\partial}^{\mu}f(z)\right)\left(\partial^{\nu}\bar{\partial}^{\nu+\lambda}g(z)\right) + o(t^k)\]
as asserted.
\end{proof}

\begin{rem} \label{remark}
In fact, Proposition \ref{thm_BaCo} is also valid for polynomials (see \cite[Lemma 14]{Bauer}). Therefore, instead of computing the double integral by brute force, we could observe that it is equal to $\pi^{2n}[(z^{\alpha}\bar{z}^{\beta})^{(1)} \sharp_1 (z^{\gamma}\bar{z}^{\eps})^{(1)}](0)$. Then, using the exact formula for the heat transform of a polynomial (i.e.~an exact version of Lemma \ref{lem1}) and the definition of $\sharp_1$ (Definition \ref{definition}), yields
\begin{align*}
[(z^{\alpha}\bar{z}^{\beta})^{(1)} \sharp_{(1)} (z^{\gamma}\bar{z}^{\eps})^{(1)}](0) &= \sum\limits_{\lambda \in \N_0^n}\sum\limits_{\mu \in \N_0^n}\sum\limits_{\nu \in \N_0^n} \left. \frac{(-1)^{\abs{\lambda}}}{\lambda !\mu ! \nu !}(\partial^{\lambda}\partial^{\mu}\bar{\partial}^{\mu}z^{\alpha}\bar{z}^{\beta})(\partial^{\nu}\bar{\partial}^{\lambda}\bar{\partial}^{\nu}z^{\gamma}\bar{z}^{\eps})\right|_{z = 0}\\
&= \sum\limits_{\lambda \in \N_0^n}\sum\limits_{\mu \in \N_0^n}\sum\limits_{\nu \in \N_0^n} \frac{(-1)^{\abs{\lambda}}}{\lambda !\mu ! \nu !} \alpha !\beta ! \gamma ! \eps ! \delta_{\alpha,\mu+\lambda}\delta_{\beta,\mu}\delta_{\gamma,\nu}\delta_{\eps,\nu+\lambda}\\
&= (-1)^{\abs{\alpha}-\abs{\beta}}\frac{\alpha !\eps !}{(\alpha-\beta)!}\delta_{\alpha - \beta,\eps - \gamma}
\end{align*}
directly.
\end{rem}

Now we are able to prove Theorem \ref{mthm} by strong induction over $k$. The case $k = 0$ is the following result of Bauer and Coburn:

\begin{prop} \label{thm_BaCoHa}
(\cite[Theorem A]{BaCo})\\
If $f,g \from \C^n \to \C$ are bounded and uniformly continuous, then
\[\lim\limits_{t \to 0} \norm{T_f^{(t)}T_g^{(t)} - T_{fg}^{(t)}} = 0.\]
\end{prop}

Now the rest is just bookkeeping.

\begin{proof}[Proof of Theorem \ref{mthm}]
We will prove
\begin{equation} \label{Induction_Hypothesis}
T_f^{(t)}T_g^{(t)} = \sum\limits_{\abs{\alpha} \leq l} \frac{(-t)^{\abs{\alpha}}}{\alpha !} T_{(\partial^{\alpha}f)(\bar{\partial}^{\alpha}g)}^{(t)} + o(t^l)
\end{equation}
for all $f,g \in C^{2l}_{buc}(\C^n)$ by strong induction. For $l = 0$ we have Proposition \ref{thm_BaCoHa}. So assume that \eqref{Induction_Hypothesis} holds for $l \in \set{0, \ldots, k-1}$. We will compare the expansions of $T_{f^{(t)}}^{(t)}T_{g^{(t)}}^{(t)}$ and $T_{f^{(t)} \sharp_t g^{(t)}}^{(t)}$. Using the induction hypothesis and Lemma \ref{lem1}, we get
\begin{align*}
T_{f^{(t)}}^{(t)}T_{g^{(t)}}^{(t)} &= \sum\limits_{\abs{\alpha}+\abs{\beta} \leq k} \frac{t^{\abs{\alpha}+\abs{\beta}}}{\alpha !\beta !} T_{\partial^{\alpha}\bar{\partial}^{\alpha}f}^{(t)}T_{\partial^{\beta}\bar{\partial}^{\beta}g}^{(t)} + o(t^k)\\
&= \sum\limits_{\abs{\alpha}+\abs{\beta} = k} \frac{t^k}{\alpha !\beta !} T_{\partial^{\alpha}\bar{\partial}^{\alpha}f}^{(t)}T_{\partial^{\beta}\bar{\partial}^{\beta}g}^{(t)} + \sum\limits_{\abs{\alpha}+\abs{\beta} = k-1} \frac{t^{k-1}}{\alpha !\beta !} T_{\partial^{\alpha}\bar{\partial}^{\alpha}f}^{(t)}T_{\partial^{\beta}\bar{\partial}^{\beta}g}^{(t)} + \ldots\\
&\qquad + \sum\limits_{\abs{\alpha}+\abs{\beta} = 1} \frac{t}{\alpha !\beta !} T_{\partial^{\alpha}\bar{\partial}^{\alpha}f}^{(t)}T_{\partial^{\beta}\bar{\partial}^{\beta}g}^{(t)} + T_f^{(t)}T_g^{(t)} + o(t^k)\\
&= \sum\limits_{\abs{\alpha}+\abs{\beta} = k} \frac{t^k}{\alpha !\beta !} T_{(\partial^{\alpha}\bar{\partial}^{\alpha}f)(\partial^{\beta}\bar{\partial}^{\beta}g)}^{(t)} + \sum\limits_{\abs{\alpha}+\abs{\beta} = k-1} \frac{t^{k-1}}{\alpha !\beta !} \sum\limits_{\abs{\gamma} \leq 1} \frac{(-t)^{\abs{\gamma}}}{\gamma !} T_{(\partial^{\alpha+\gamma}\bar{\partial}^{\alpha}f)(\partial^{\beta}\bar{\partial}^{\beta+\gamma}g)}^{(t)} + \ldots\\
&\qquad + \sum\limits_{\abs{\alpha}+\abs{\beta} = 1} \frac{t}{\alpha !\beta !} \sum\limits_{\abs{\gamma} \leq k-1} \frac{(-t)^{\abs{\gamma}}}{\gamma !} T_{(\partial^{\alpha+\gamma}\bar{\partial}^{\alpha}f)(\partial^{\beta}\bar{\partial}^{\beta+\gamma}g)}^{(t)} + T_f^{(t)}T_g^{(t)} + o(t^k)\\
&= \sum\limits_{\substack{\abs{\alpha}+\abs{\beta}+\abs{\gamma} \leq k \\ \abs{\alpha} + \abs{\beta} \neq 0}} (-1)^{\abs{\gamma}}\frac{t^{\abs{\alpha}+\abs{\beta}+\abs{\gamma}}}{\alpha !\beta !\gamma !} T_{(\partial^{\alpha+\gamma}\bar{\partial}^{\alpha}f)(\partial^{\beta}\bar{\partial}^{\beta+\gamma}g)}^{(t)} + T_f^{(t)}T_g^{(t)} + o(t^k).
\end{align*}
Hence, using Proposition \ref{thm_BaCo} and comparing with Lemma \ref{lem2}, we get
\[T_f^{(t)}T_g^{(t)} = \sum\limits_{\abs{\gamma} \leq k} (-1)^{\abs{\gamma}}\frac{t^{\abs{\gamma}}}{\gamma !} T_{(\partial^{\gamma}f)(\bar{\partial}^{\gamma}g)}^{(t)} + o(t^k)\]
and we are done.
\end{proof}

\section{Proof of Theorem 3}

Recall that we want to prove
\[\sum\limits_{\abs{\alpha} \leq k} \frac{(-t)^{\abs{\alpha}}}{\alpha !} \left((\partial^{\alpha}f)(\bar{\partial}^{\alpha}g)\right)^{(t)} = \sum\limits_{\abs{\alpha} \leq k} \frac{t^{\abs{\alpha}}}{\alpha !} (\bar{\partial}^{\alpha}f^{(t)})(\partial^{\alpha}g^{(t)}) + o(t^k)\]
as $t \to 0$. We will need the following combinatorial identity.

\begin{lem} \label{lem3}
Let $q \geq p \geq 0$ and $l \geq 0$. Then
\[\sum\limits_{m = 0}^p (-1)^m \frac{(q-m+l)!}{m!(p-m)!(q-m)!} = \begin{cases} 0 & \text{if } l < p, \\ \frac{l!(l+q-p)!}{q!p!(l-p)!} & \text{if } l \geq p. \end{cases}\]
\end{lem}

\begin{proof}
Consider the equality
\[\sum\limits_{m = 0}^p \frac{p!}{m!(p-m)!} (-1)^mx^{q-m+l} = (x-1)^px^{q-p+l},\]
differentiate it $l$ times and evaluate at $x = 1$.
\end{proof}

\begin{proof}[Proof of Theorem \ref{mthm2}]
Taylor expanding the left-hand side and using the same arguments as in Lemma \ref{lem1} yields
\begin{align} \label{eq_mthm2}
\sum\limits_{\abs{\alpha} \leq k} \frac{(-t)^{\abs{\alpha}}}{\alpha !} \left((\partial^{\alpha}f)(\bar{\partial}^{\alpha}g)\right)^{(t)} &= \sum\limits_{\abs{\alpha} \leq k} \frac{(-t)^{\abs{\alpha}}}{\alpha !}  \frac{1}{\pi^n} \int_{\C^n} (\partial^{\alpha}f)(\sqrt{t}w + z)(\bar{\partial}^{\alpha}g)(\sqrt{t}w + z)e^{-\abs{w}^2} \, \mathrm{d}w\notag\\
&= \frac{1}{\pi^n} \sum\limits_{\abs{\alpha} \leq k} \frac{(-t)^{\abs{\alpha}}}{\alpha !} \sum\limits_{\abs{\beta}+\abs{\gamma}+\abs{\eps}+\abs{\zeta} \leq 2(k-\abs{\alpha})} \frac{t^{\frac{\abs{\beta}+\abs{\gamma}+\abs{\eps}+\abs{\zeta}}{2}}}{\beta !\gamma !\eps !\zeta !} (\partial^{\alpha+\beta}\bar{\partial}^{\gamma}f)(z)\notag\\
&\qquad \qquad \qquad \qquad \qquad \cdot (\partial^{\eps}\bar{\partial}^{\alpha+\zeta}g)(z) \int_{\C^n} w^{\beta+\eps}\bar{w}^{\gamma+\zeta}e^{-\abs{w}^2} \, \mathrm{d}w + o(t^k)\notag\\
&= \sum\limits_{2\abs{\alpha}+\abs{\beta}+\abs{\gamma}+\abs{\eps}+\abs{\zeta} \leq 2k} (-1)^{\abs{\alpha}}\frac{t^{\frac{2\abs{\alpha}+\abs{\beta}+\abs{\gamma}+\abs{\eps}+\abs{\zeta}}{2}}}{\alpha !\beta !\gamma !\eps !\zeta !}\\
&\qquad \qquad \qquad \quad \cdot (\partial^{\alpha+\beta}\bar{\partial}^{\gamma}f)(z)(\partial^{\eps}\bar{\partial}^{\alpha+\zeta}g)(z) (\gamma+\zeta)!\delta_{\beta+\eps,\gamma+\zeta} + o(t^k).\notag
\end{align}
Fix $j \in \set{1, \ldots, n}$ and $p,q \in \N$ with $q \geq p$. Then Lemma \ref{lem3} implies (with $\alpha_j = m$, $\gamma_j = l$)
\begin{align*}
\sum\limits_{\alpha_j+\beta_j = p}\sum\limits_{\alpha_j+\zeta_j = q} \frac{(-1)^{\alpha_j}(\gamma_j+\zeta_j)!}{\alpha_j!\beta_j!\gamma_j!\eps_j!\zeta_j!}\delta_{\beta_j+\eps_j,\gamma_j+\zeta_j} &= \sum\limits_{m = 0}^p \frac{(-1)^m(\gamma_j+q-m)!}{m!(p-m)!\gamma_j!\eps_j!(q-m)!}\delta_{p-m+\eps_j,\gamma_j+q-m}\\
&= \begin{cases} 0 & \text{if } \gamma_j < p \\ \frac{(\gamma_j+q-p)!}{q!p!(\gamma_j-p)!\eps_j!}\delta_{p+\eps_j,\gamma_j+q} & \text{if } \gamma_j \geq p \end{cases}\\
&= \begin{cases} 0 & \text{if } \gamma_j < p \\ \frac{1}{q!p!(\gamma_j-p)!}\delta_{p+\eps_j,\gamma_j+q} & \text{if } \gamma_j \geq p. \end{cases}
\end{align*}
As the expression on the left-hand side is symmetric with respect to $\beta_j \leftrightarrow \zeta_j$, $\gamma_j \leftrightarrow \eps_j$, $p \leftrightarrow q$, we also obtain
\[\sum\limits_{\alpha_j+\beta_j = p}\sum\limits_{\alpha_j+\zeta_j = q} \frac{(-1)^{\alpha_j}(\gamma_j+\zeta_j)!}{\alpha_j!\beta_j!\gamma_j!\eps_j!\zeta_j!}\delta_{\beta_j+\eps_j,\gamma_j+\zeta_j}  = \begin{cases} 0 & \text{if } \eps_j < q \\ \frac{1}{p!q!(\eps_j-q)!}\delta_{q+\gamma_j,\eps_j+p} & \text{if } \eps_j \geq q \end{cases}\]
for $q \leq p$, which is actually exactly the same as above due to the Kronecker delta. Now the idea is to group the terms in \eqref{eq_mthm2} with fixed $p$ and $q$. In particular, the terms with $\alpha_j+\beta_j > \gamma_j$ or $\alpha_j+\zeta_j > \eps_j$ for some $j \in \set{1, \ldots, n}$ can be removed as they sum to $0$. We therefore set $\rho := \alpha + \beta$, $\sigma := \alpha  +\zeta$, $\mu := \gamma - \alpha - \beta$ and $\nu := \eps - \alpha - \zeta$. This yields
\begin{align*}
\sum\limits_{\abs{\alpha} \leq k} \frac{(-t)^{\abs{\alpha}}}{\alpha !} \left((\partial^{\alpha}f)(\bar{\partial}^{\alpha}g)\right)^{(t)} &= \sum\limits_{2\abs{\rho}+2\abs{\sigma}+\abs{\mu}+\abs{\nu} \leq 2k} \frac{t^{\frac{2\abs{\rho}+2\abs{\sigma}+\abs{\mu}+\abs{\nu}}{2}}}{\rho !\sigma !\mu!} (\partial^{\rho}\bar{\partial}^{\rho+\mu}f)(z) (\partial^{\nu+\sigma}\bar{\partial}^{\sigma}g)(z) \delta_{\mu,\nu}\\
&\qquad \qquad \qquad \qquad \qquad + o(t^k)\\
&= \sum\limits_{\abs{\rho}+\abs{\sigma}+\abs{\mu} \leq k} \frac{t^{\abs{\rho}+\abs{\sigma}+\abs{\mu}}}{\rho !\sigma !\mu!} (\partial^{\rho}\bar{\partial}^{\rho+\mu}f)(z) (\partial^{\mu+\sigma}\bar{\partial}^{\sigma}g)(z) + o(t^k)\\
&= \sum\limits_{\abs{\rho}+\abs{\sigma}+\abs{\mu} \leq k} \frac{t^{\abs{\mu}}}{\mu !}\bar{\partial}^{\mu}\left(\frac{t^{\abs{\rho}}}{\rho !}\partial^{\rho}\bar{\partial}^{\rho}f\right)(z) \partial^{\mu}\left(\frac{t^{\abs{\sigma}}}{\sigma !}\partial^{\sigma}\bar{\partial}^{\sigma}g\right)(z) + o(t^k)\\
&= \sum\limits_{\abs{\mu} \leq k} \frac{t^{\abs{\mu}}}{\mu !} (\bar{\partial}^{\mu}f^{(t)})(\partial^{\mu}g^{(t)}) + o(t^k)
\end{align*}
by Lemma \ref{lem1}.
\end{proof}

\section{Open problems}

Here we summarize some related problems that we were not able to solve here and probably go beyond the methods of this paper. They will be considered in future research.

\paragraph{(1) Find the least number of derivatives needed for the product expansion.} \cite{BaCoHa} suggests that the regularity assumptions in Theorem \ref{mthm} and Corollary \ref{cor} are not quite optimal. However, it is not so clear whether $2k$ derivatives are actually needed. For the case $k = 0$ not even continuity is needed, but the oscillation needs to be constrained (see \cite{BaCoHa}). We conjecture that something similar happens for $k \geq 1$. However, this requires a more sophisticated approach than presented here because even though all the higher derivatives magically disappear in the end, they are needed to bound the error term.

\vspace{-0.3cm}
\paragraph{(2) Find an example of two functions with only one bounded derivative where the statement of Corollary \ref{cor} is wrong.} As we expect that only one derivative (without regularity assumptions on the partial derivatives) is not enough, a counterexample needs to be found.

\vspace{-0.3cm}
\paragraph{(3) Show that Theorem \ref{mthm2} still holds with less than $2k$ derivatives or find a counterexample.} As this identity is on a purely functional level, this might be easier than \textbf{(1)} or \textbf{(2)}.

\vspace{-0.3cm}
\paragraph{(4) Prove a similar result for the Bergman space.} A similar result is probably also true for the Bergman space on the unit ball. However, a different approach is needed because an analogue of Proposition \ref{thm_BaCo} does not exist and in general it appears to be less likely for a product of two Toeplitz operators to be a Toeplitz operator again.

\section*{Acknowledgements} I would like to thank the anonymous referee for pointing out an error in the original manuscript and for some other valuable suggestions. I would also like to thank Wolfram Bauer for pointing out Remark \ref{remark} and other simplifications.

\bigskip

\noindent
Raffael Hagger\\
Institut f\"ur Analysis\\
Leibniz Universit\"at Hannover\\
Welfengarten 1\\
30167 Hannover\\
GERMANY\\
raffael.hagger@math.uni-hannover.de

\end{document}